\documentclass[11pt]{amsart}
\usepackage{latexsym}
\usepackage{graphics}
\usepackage{graphicx}
\usepackage{epstopdf}
\usepackage{tikz}
\usepackage{amsmath}
\usepackage{hyperref}

\usepackage{caption}
\usepackage[labelformat=simple]{subcaption}

\newtheorem{theorem}{Theorem}[section]

\newtheorem{lemma}[theorem]{Lemma}
\newtheorem{corollary}[theorem]{Corollary}

\theoremstyle{remark}
\newtheorem{remark}[theorem]{Remark}

 \title{The number of surfaces of fixed genus in an alternating link complement}
 \author{Joel Hass } 
 \thanks{Partially supported by NSF grant 1117663, BSF grant 2012188, and the Ambrose Monell Foundation}
\address{Department of Mathematics, University of California, Davis
California 95616 \& School of Mathematics, Institute for Advanced Study, Princeton NJ 08540}
\email{hass@math.ucdavis.edu}

\author{Abigail Thompson}
\thanks{Partially supported by NSF grant DMS-1207765; Neil Chriss and Natasha Herron Chris Founders' Circle Member,  IAS 2015-2016}
\address{Department of Mathematics, University of California, Davis
California 95616 \& School of Mathematics, Institute for Advanced Study, Princeton NJ 08540}
\email{thompson@math.ucdavis.edu}

 \author{Anastasiia Tsvietkova} 
 \thanks{Partially supported by NSF grant DMS-1406588}
\address{Department of Mathematics, University of California, Davis
California 95616}
  \email{tsvietkova@math.ucdavis.edu}

 \date{}
 \subjclass[2010]{}
 \everymath{\displaystyle}
 
\begin{document}
 
  \footnotesize
  \begin{abstract} 
Let $L$ be a prime alternating link with $n$ crossings. We show that for each fixed $g$, the number of genus $g$ incompressible surfaces in the complement of $L$ is bounded by an explicitly given polynomial in $n$. 
Previous bounds were exponential in $n$.
  \end{abstract}
  
 \maketitle
 \normalsize
 
 \section{Overview}
 Let $L$ be a non-split prime alternating link with an $n$-crossing diagram and let  $M_L$ denote the complement of $L$ in $S^3$. In this paper we give a bound on the number of isotopy classes of closed incompressible surfaces of genus $g$  embedded in $M_L$. For each choice of genus $g$, this bound is a polynomial in $n$.  For example, we show that the number of genus two incompressible surfaces in $M_L$ is at most $12n^3$.  More generally, we show that the number of genus $g$ incompressible surfaces in the complement of an $n$-crossing alternating diagram is at most $ C_g n^{40g^2}$, where $C_g$ is a constant depending only on the genus.
 
 The surfaces we consider are closed and incompressible, but not necessarily disjoint. The number of disjoint incompressible surfaces in a manifold
 is much easier to bound, as  originally observed by Kneser \cite{Kneser:29}. Kneser showed that the number of such surfaces is bounded by a linear function of the number of tetrahedra $t$ required to triangulate the manifold.  For link complements,
 the number $t$ is itself a linear multiple of the number of crossings in a link diagram $n$.
 
 In non-hyperbolic manifolds there can be infinitely many distinct incompressible surfaces of a fixed genus, as for example in the 3-torus, which has infinitely many non-isotopic essential tori.
 In a hyperbolic manifold the number of such surfaces is always finite, as can be seen by isotoping each surface to a least area representative and applying
 the Gauss-Bonnet Theorem and Schoen's curvature estimates \cite{Schoen, Hass:95}. This argument applies also to $\pi_1$-injective immersions, but  
 is not constructive and gives no explicit bound on the number of surfaces of a given genus.
 
 One approach to counting the number of embedded incompressible surfaces of genus $g$ is through normal surface theory.  
 Each incompressible surface can be isotoped to be normal, and can then be expressed as a sum  of
 certain fundamental normal surfaces. However this process leads to an exponential bound on the number
 of incompressible surfaces of genus $g$, either in terms of the number of tetrahedra in a triangulation $t$, or
 in terms of the crossing number $n$.  The main issue is that the number of fundamental surfaces of a given genus, and even
 the number of vertex fundamental surfaces, can be exponential
 in $t$ \cite{HassLagariasPippenger}.  A second difficulty in applying normal surfaces
  is that an incompressible surface may not be fundamental, so that 
 one must also count all Haken sums of normal surfaces that can result in a given genus \cite{Haken:61}.
 
Bounds for the number of  immersed $\pi_1$-injective surfaces in a closed hyperbolic 3-manifold, up to homotopy,
have been obtained by Masters and by Kahn and Markovic.
Masters showed that the number of surface subgroups in a closed hyperbolic 
 3-manifold $M$, up to conjugacy, is at most $g^{Cg}$ for some $C$ \cite{Masters}.
Kahn and Markovic  showed that this number is bounded below by  $(cg)^{2g}$ and above by  $(Cg)^{2g}$ for constants $c,C$ depending on $M$ \cite{KahnMarkovic}. 
These papers consider the number of immersed surfaces in a fixed closed manifold as  the genus grows. 
In contrast, we obtain bounds on the number of embedded surfaces in terms of both the genus of the surface and the complexity of the hyperbolic manifold. Our bounds apply to a class of  link complements, as opposed to the closed manifolds previously considered.

\section{Standard position for a surface}\label{IntersectionPattern}
   
We first review techniques of Menasco that place an incompressible surface  in $M_L$ into a standard position with respect to a projection plane $Q$ for $L$  \cite{Menasco1984}.   
We begin by isotoping the link $L$ so that its diagram is alternating and reduced, and  so that it lies in the projection 
plane  $Q$ with the exception of two small arcs near each crossing,
 one of which drops below $Q$, and one of which rises above it. 
 $L$ then lies on a union of two overlapping 2-spheres in $S^3$, $S_+$ and $S_-$, which agree with $Q$ 
 except along bubbles surrounding each crossing. At the bubbles, $S_+$ and $S_-$ go over the top and bottom hemispheres of each bubble, respectively. 
We denote by $B_+$ and $B_-$ the balls in $S^3$ lying respectively above  $S_+$ and below $S_-$.

We then compress along curves that are parallel to a meridian of the link.
A {\em meridianal compression} on $F_0$ surgers a curve in $F_0$,
 parallel in $M_L$  to a meridian of $L$, and creates a new pair  of meridianal punctures. 
 
A surface is  {\em meridianally incompressible} if no meridianal surgery can be performed.  
Menasco showed in \cite{Menasco1984} how to isotope a meridianally incompressible surface into {\em standard position} with respect to $S_+$ and $S_-$.
The properties of a surface in standard position, which we call  $F$,  are summarized in Lemma~\ref{properties}.
  
We can assign to each curve  $C$ in $F  \cap S_+$ or  $F \cap  S_-$  a word in the letters $P$ and $S$,  defined up to cyclic order, with $P$ indicating
 a point where $C$ crosses a strand of the link along a puncture and $S$ indicating that the curve passes through a bubble region adjacent to a saddle of $F$.
 Figures~\ref{Figure1.1}, \ref{Figure1.2} and \ref{Figure1.3} depict a component of $F \cap S_+$ giving an $SSSS$,  $SPPPPP$  and $PPPS$ curve on $S_+$.

We define a complexity  $|F| = p+s+c$ where $p$ is the sum of the number of $P$'s associated to all curves in $F \cap S_+$, 
$s$ is the sum of the number of $S$'s associated to these words, and $c$ is the sum of the number of curves in  $ F  \cap  S_+$ 
plus the number of curves in $  F  \cap S_-$. If $F$ minimizes this sum among  standard position 
surfaces in its isotopy class, then $F$ is then said to be in {\em $|F|$ minimizing standard position}. 
 
 \begin{lemma}\label{properties} 
Suppose an incompressible and meridianally incompressible surface $F$ is in $|F|$ minimizing standard position.
Then the curves of $F\cap S_+$ and  $F\cap S_-$ and the associated words in  the letters $P,S$ satisfy the following properties: 
 \begin{enumerate}
 \item \label{Disks}  Each curve of $F \cap S_+$ bounds a disk in $F \cap B_+$, and similarly each  curve of $F \cap S_-$ bounds a disk in $F \cap B_-$ .
\item \label{SaddleTwice}  No curve passes through the same saddle twice.
 \item \label{Innermost} An innermost curve of  $F \cap S_+$  or  $F\cap S_-$ does not go through two successive saddles.
 \item \label{Balance} An equal number of curves pass through each side of a saddle.
 \item \label{ArcTwice} No curve has two successive punctures on the same arc of $L$ with no intermediate saddles.
 \item \label{SaddleArc} No curve passes through a saddle and then crosses an arc of $L$ adjacent to the saddle.
  \item \label{NoSSSS}  No word has the form  $P^i S ^j$ for $j>0$. 
 \item \label{TwoPs} Each word contains at least two $P$'s.
 \item \label{EvenLength4} Each word has length at least four.
 \end{enumerate}
 \end{lemma}
 
\begin{proof}
\begin{figure}
\centering
\begin{subfigure}[b]{0.3 \textwidth}
\centering
\includegraphics[scale=0.33]{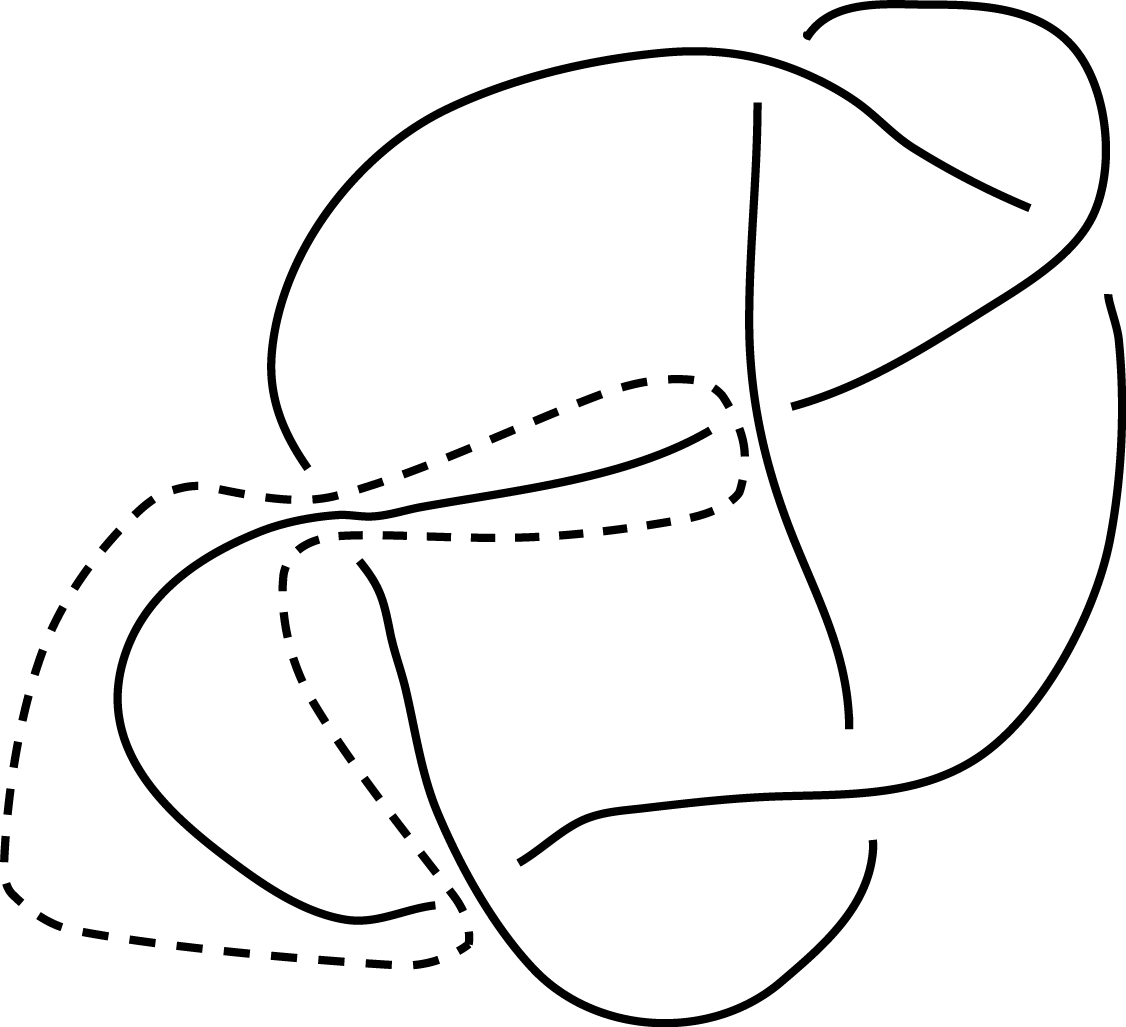}
\caption{}
\label{Figure1.1}
\end{subfigure}
\begin{subfigure}[b]{0.3 \textwidth}
\centering
\includegraphics[scale=0.33]{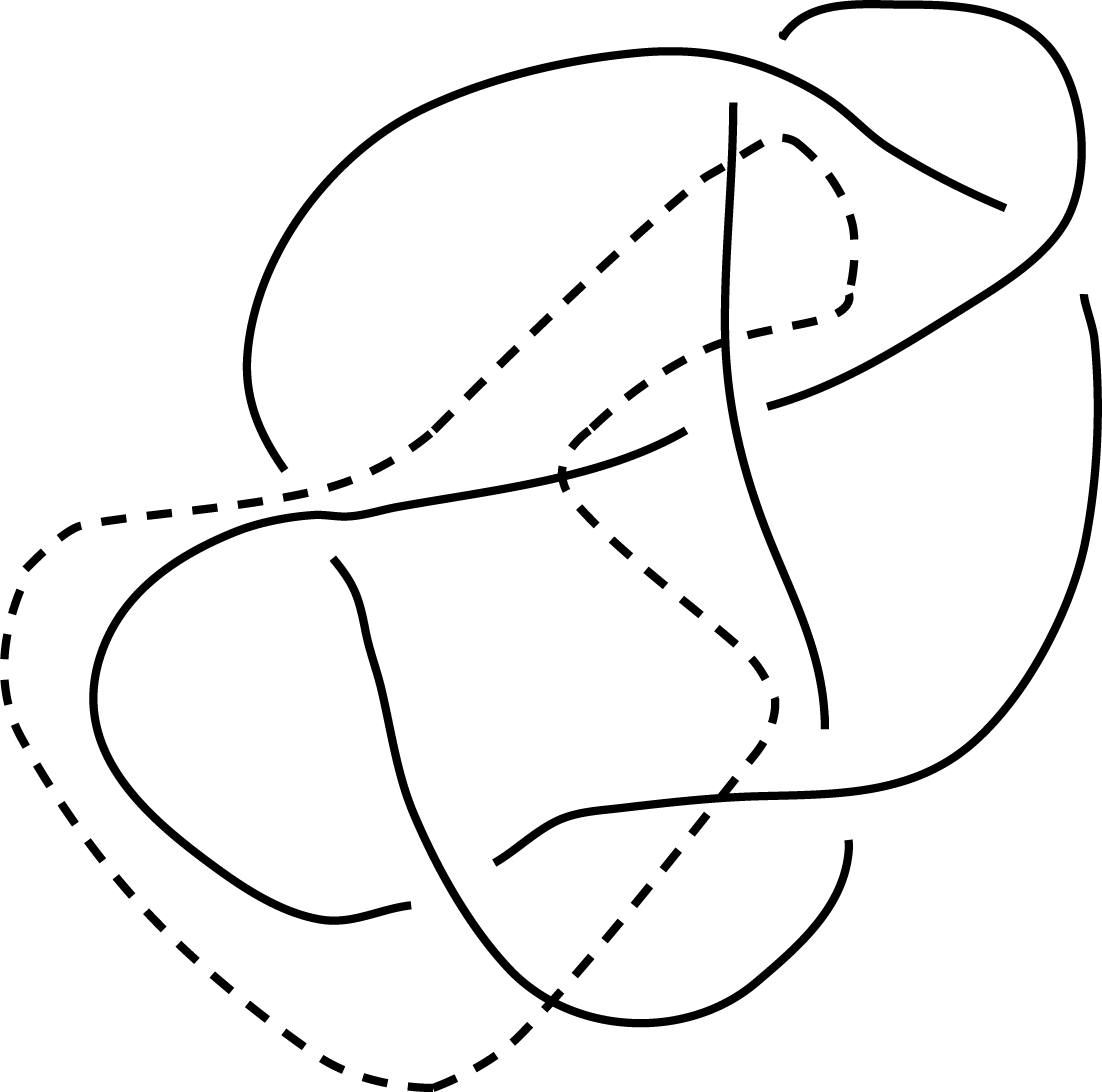}
\caption{}
\label{Figure1.2}
\end{subfigure}
\begin{subfigure}[b]{0.3 \textwidth}
\centering
\includegraphics[scale=0.33]{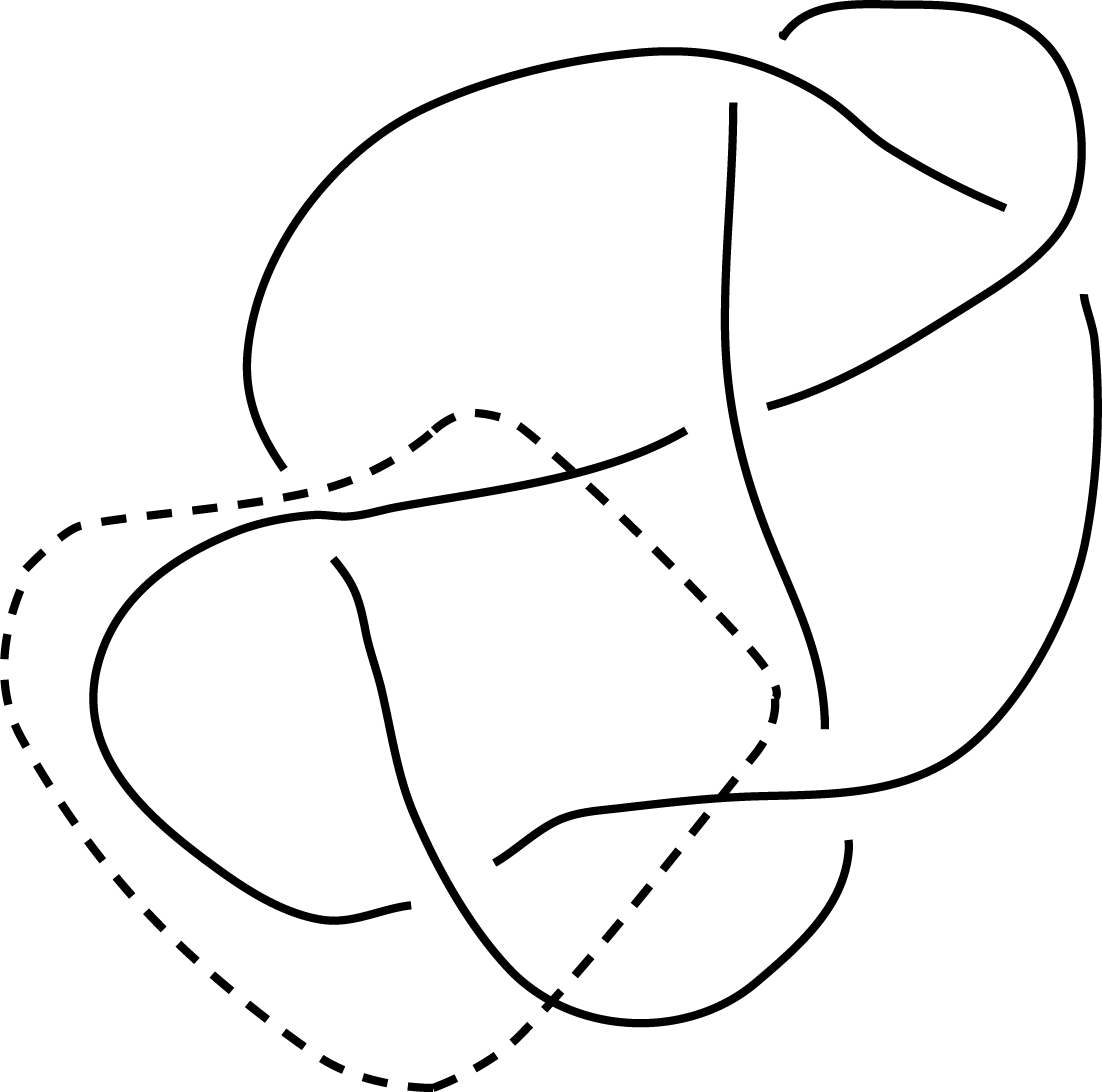}
\caption{}
\label{Figure1.3}
\end{subfigure}
\caption{The link $L$ and a curve from $F\cap S_\pm$}
\end{figure}

Property   (\ref{Disks}) follows from the incompressibility of $F$ and that fact that the complexity   $|F| = p+s+c$ is minimized.
 Figures~\ref{Figure1.1}, \ref{Figure1.2} and \ref{Figure1.3}  show curves that
are ruled out by (\ref{SaddleTwice}), (\ref{ArcTwice}), (\ref{SaddleArc}) respectively. 
Properties (\ref{SaddleTwice}), (\ref{ArcTwice}), (\ref{SaddleArc}) are proven in Lemma 3.2 of \cite{HassMenasco}.
Property   (\ref{Innermost})  holds since  successive saddles lie on opposite sides of a curve, and thus
cannot exist for an innermost curve. Property  (\ref{Balance}) follows from the fact that each saddle of a surface results in one
intersection curve on each side of a crossing. The proofs of (\ref{NoSSSS}) and (\ref{TwoPs}) are given in Lemma 2 of \cite{Menasco1984}. 
For (\ref{EvenLength4}), first note that any collection of closed curves on a plane intersect in an even number of points, so that $w$ has an even number of letters. The words $PS$  and $SS$ are not possible by (\ref{NoSSSS})  and $PP$ is ruled out by (\ref{ArcTwice})
and the fact that the link is prime.

  \end{proof}
  
 Lemma~\ref{properties} severely restricts the types of surfaces in a non-split prime alternating link complement.  For example, Property~\ref{TwoPs} implies that
 there are no closed incompressible and meridianally incompressible surfaces in the complement of $L$. This was used by Menasco to show that there are no totally geodesic surfaces in such a link complement.
 
 \section{Counting genus two surfaces}
 We first bound the number of  genus two surfaces in $M_L$.
Lemma~\ref{GenusTwoWords}  follows from the arguments in Menasco  \cite{Menasco1984}.  We give the argument for completeness.
  
 \begin{lemma} \label{GenusTwoWords}  Suppose $F_0$ is a closed incompressible  genus two surface in $M_L$. 
 Let $F$ be the result of a maximal number of meridianal compressions on $F_0$.
 Then $F$ is a four-punctured sphere.
 Furthermore, when placed in  $|F|$ minimizing standard position relative to $S_+ \cup S_-$, 
 $F$ intersects $S_+$ in either a single $PPPP$ curve or in  two $PSPS$ curves. 
 \end{lemma}
 \begin{proof} 
Meridianally compress $F_0$ to obtain the meridianally incompressible surface $F$, and place $F$ in 
 $|F|$ minimizing standard position.  Since $F_0$ is incompressible, $F$ intersects $S_+$ in at least one curve. 
 The word associated to every intersection curve has length at least four, and has at least two $P$'s by Lemma \ref{properties} (\ref{TwoPs}) and (\ref{EvenLength4}). 
By Lemma~\ref{properties} (\ref{NoSSSS}), every word is either of the form $PSPS$, or $PPPP$, or has length at least six.
Since $F_0$ has genus two, there are four meridianal punctures in $F$, and we have the following cases.
 
 1) If  $F$ has zero or two punctures, then the intersection pattern with $S_+$ has only two $P$'s.  
 Then there can then be only a single intersection curve which must have the form $PSPS$. 
 But a saddle through which this curve passes has a distinct curve passing through its opposite side, 
 so at least four $P$'s appear in curves of $F \cap S_+$.
 
 2) If $F$ is a 4-punctured sphere, then every curve in $F \cap S_+$ has associated word of the form 
 $PPPP$ or $PSPS$, with at least two $P$'s. 
Since $L$ is alternating, an innermost curve cannot have two consecutive saddles not separated by punctures. 
Moreover, if it contains a saddle, then  $F \cap S_+$ contains a second innermost curve. 
Since just four $P$'s are available, there is either a single $PPPP$ curve which is innermost on both sides
or each of two  innermost curves must have associated 
word  $PSPS$ and this accounts for all $P$'s and therefore all curves.
\end{proof}

We now count the number of such curves.
Suppose $F_0$ is a closed incompressible  genus two surface in $M_L$, and $F$ is obtained from $F_0$ by a sequence of two meridianal compressions
and that $F$ is in standard position relative to $S_+ \cup S_-$ with $|F|$ minimized.

\begin{lemma}\label{genus2curvecount} 
The number of  isotopy classes of curve configurations of $F \cap S_+$ with $F$ arising from meridianal compressions of a genus two incompressible surface in $M_L$ is less than  $2n^3$.
\end{lemma}
\begin{proof}
There are $2n$ arcs in an $n$-crossing link diagram.
For a $PPPP$ word, the last puncture is determined by the location of the other three. Indeed, suppose two distinct $PPPP$ curves, $c_1$ and $c_2$, have exactly three punctures that coincide. Then one can form a closed curve going through the other two punctures, and corresponding to a $PP$ word. But such a curve  in a reduced alternating diagram of a prime link is trivial, and therefore $c_1$ is isotopic to $c_2$.
A cyclic reordering of the initial edge punctured gives the same curve, so we divide by four to get the number of configurations. 
Hence the number of $PPPP$ curves is less than $ (2n) (2n-1)(2n-2)/4 = 2n^3-3n^2 +n $, 
 obtained by picking three successive edges of the diagram to cross. 

We now bound the number of  $PSPS$ curves.
There are $n$ crossings, each with two sides through which a curve can pass and contribute an $S$ to a word.
Suppose two curves emerge from the same side of a crossing going in the same direction, 
cross distinct punctures, and then enter a second saddle
 on the same side of a second crossing. 
Then there is a closed curve intersecting the diagram in two punctures only, formed by joining the two
arcs after they leave the first saddle and again just before they enter the second.  This loop intersects the diagram in 
a $PP$ word, and since the diagram is  prime, this loop is trivial and intersects the same arc of the diagram twice.
It follows that the two original arcs are isotopic and that a $PSPS$ curve is completely 
determined by a choice of two saddles. 
One $PSPS$ curve then determines the second, since it determines the location of the second pair of saddles.
The number of $PSPS$ curves is then at most  $ {2n \choose 2} =  2n^2-n $.
Adding to the previous count gives a bound on the number of configurations of $2n^3-3n^2 +n  + 2n^2-n   = 2n^3-n^2  <  2 n^3$.
\end{proof}

\begin{theorem}\label{genus2count} 
The number of closed incompressible genus two surfaces in $M_L$ is less than  $12n^3$.
\end{theorem}
\begin{proof}
A genus two surface is obtained from a meridianally incompressible surface by tubing together pairs of boundary components.
By Lemma \ref{GenusTwoWords}, every closed incompressible surface meridianally compresses
to a surface with four boundary punctures. 

We obtain a closed surface from a surface with meridian boundary curves
by a  tubing procedure that joins punctures on the same component of $L$ in pairs.
These pairs of punctures cannot be interleaved as one traverses once around a component of  the link,
as otherwise the tubed surface could not be embedded.
If $L$ is a knot, there are two ways to partition the 4 punctures into pairs with no interleaving,
and each of these  gives rise to three pairs of tubes.
Figure~\ref{Tubing1} shows  a knot $L$ and a four-punctured sphere in $M_L$. Figures \ref{Tubing2}--\ref{Tubing7} show the six possible tubings. This is the $n=2$ case of the general formula for the number of ways to form $n$ tubes starting with $2n$ curves, which equals ${2n \choose n}$ \cite{Mossessian}. 

If $L$ is a link with more than one component, then each component of $L$ meets at most four punctures. 
If two components meet two punctures each, then there are two ways to 
undo a meridianal boundary compression and 
obtain a closed surface for each of the components, giving a total of four choices.
If one component of $L$ meets all four punctures, 
there are at most six ways to add tubes to $F$ as before.

Lemma~\ref{genus2curvecount} then implies that the number of closed incompressible surfaces of genus two in $M_L$ is less than $6(  2 n^3) = 12n^3$.
\end{proof}

\begin{figure}
\centering
\begin{subfigure}[b]{0.24 \textwidth}
\centering
\includegraphics[scale=0.5]{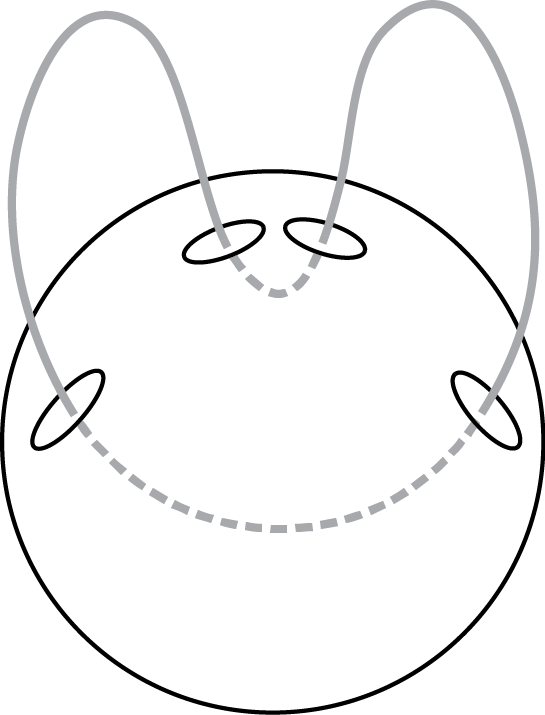}
\caption{}
\label{Tubing1}
\end{subfigure}
\begin{subfigure}[b]{0.24 \textwidth}
\centering
\includegraphics[scale=0.5]{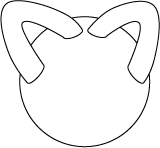}
\caption{}
\label{Tubing2}
\end{subfigure}
\begin{subfigure}[b]{0.24 \textwidth}
\centering
\includegraphics[scale=0.5]{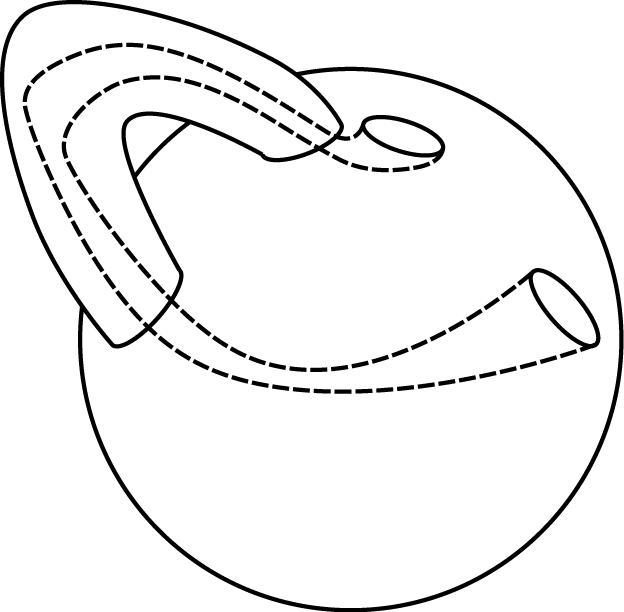}
\caption{}
\label{Tubing3}
\end{subfigure}
\begin{subfigure}[b]{0.24 \textwidth}
\centering
\includegraphics[scale=0.5]{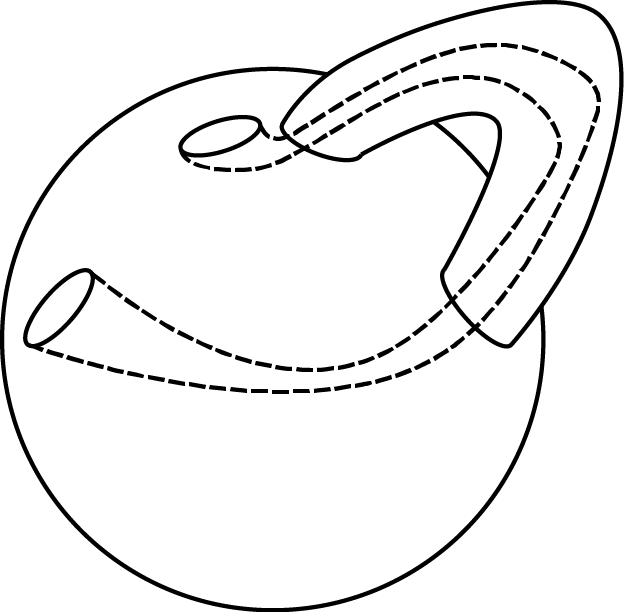}
\caption{}
\label{Tubing4}
\end{subfigure}
\begin{subfigure}[b]{0.24 \textwidth}
\centering
\includegraphics[scale=0.5]{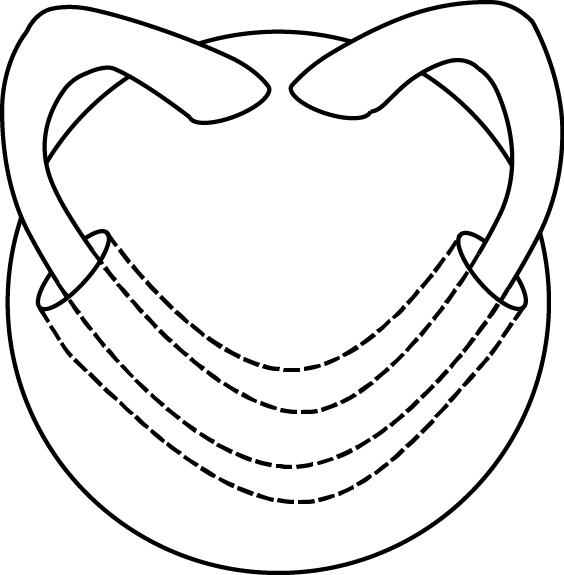}
\caption{}
\label{Tubing5}
\end{subfigure}
\begin{subfigure}[b]{0.24 \textwidth}
\centering
\includegraphics[scale=0.5]{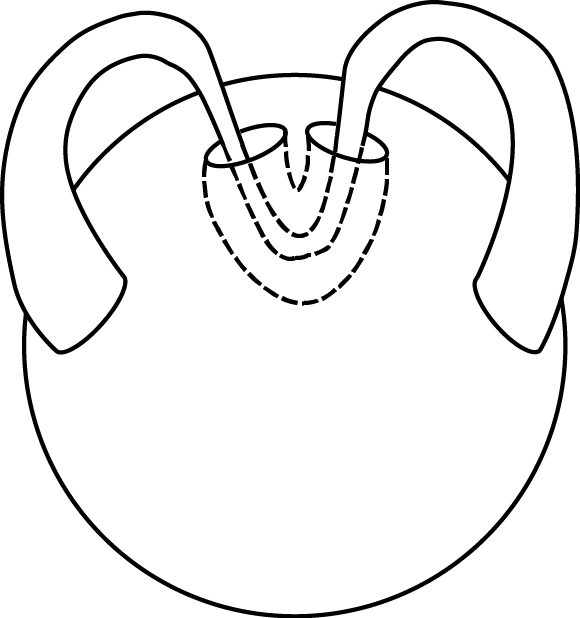}
\caption{}
\label{Tubing6}
\end{subfigure}
\begin{subfigure}[b]{0.24 \textwidth}
\centering
\includegraphics[scale=0.5]{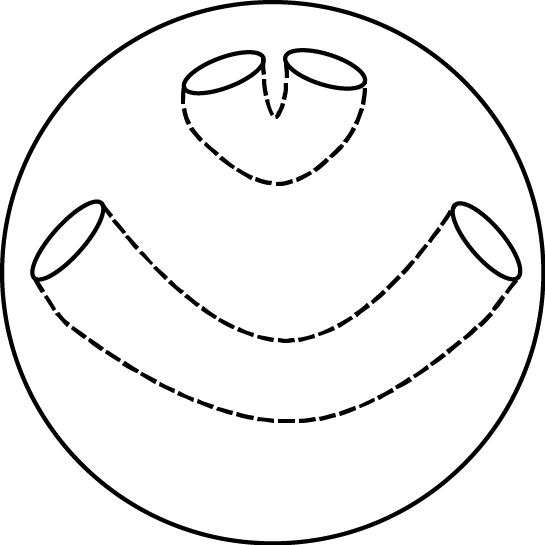}
\caption{}
\label{Tubing7}
\end{subfigure}
\caption{Six ways to add two tubes to a four-punctured sphere in $M_L$}
\end{figure}
\label{Tubing}

\section{Decomposing a surface into polygons}\label{CountingArgument}

Let $F_0 \subset M_L$ be an incompressible  genus $g$ surface and let $F$ be a meridianally incompressible surface obtained from $F_0$.   Place  $F$ in standard position relative to $S_+ \cup S_-$ with $|F|$ minimized.
In this section, we  bound the number of curves in $F\cap{S_+}$ and we bound the maximum length of the word associated to each curve.
The  method is based on studying a decomposition of $F$ 
into polygons given by its intersections with $S_+$ and $S_-$ .

 \begin{lemma}\label{numcompressions} 
 Suppose that $F_0$ is an embedded genus $g$ incompressible surface in $M_L$ and that $F$ is obtained by
 meridianal compression on $F_0$ along pairwise non-parallel curves. 
Then the number of meridianal compressions is at most $2g-2$ and the number of meridianal boundary
curves of $F$ is at most $4g-4$.
\end{lemma}
\begin{proof}
The compressing curves in $F_0$ cut $F_0$ into a collection of complementary surfaces $F_i$.
Each $F_i$ has an even number of boundary curves, since 
meridianal compression results in a surface that meets $L$ in an even number of punctures.
No component $F_i$ an annulus, since the meridianal compressions occur along non-parallel curves.
Thus the Euler characteristic of each $F_i$ is negative and even.
A surface with negative even Euler characteristic can be divided into four-punctured spheres, each with Euler characteristic
-2.  Since their Euler characteristics add to the Euler characteristic of $F_0$, equal to $2-2g$, the number of resulting
4-punctured spheres is at most $(2-2g)/(-2) =g-1$.  Each 4-punctured sphere contributes
at most four meridianal boundary curves to $F$, and  it follows that 
the number of  meridianal boundary curves after a maximal set of meridianal compressions, is at most $4g-4$.
 \end{proof}

\begin{corollary}\label{WordsNumber} A closed genus $g$ surface  $F_0$ yields at most $2g-2$ curves in $F\cap S_+$.
\end{corollary}
\begin{proof}
From Lemma~\ref{numcompressions} we can have at most $4g-4$ punctures in $F$. 
By Lemma \ref{properties} (\ref{TwoPs}), each curve accounts for at least two $P$'s, so 
there are at most $2g-2$ curves in $F\cap S_+$ or $F\cap S_-$.
\end{proof}

We now bound the maximum length of the word associated to each curve. 
By filling in each puncture of $F$ we obtain a new surface $\bar F$ with Euler characteristic 
$\chi(\bar F) = \chi(F)+p$. 
Decompose $\bar F$ into polygons as follows. Outside the bubbles containing the saddles of $\bar F$ we take the arcs of $\bar F \cap Q$ to form part of a
graph on $\bar F$. Each saddle in $\bar F$ is a disk with four arcs of $\bar F \cap Q$ meeting its boundary, and we cone these to the center of the saddle, where we add a vertex.
The resulting graph  has valence four vertices at the center of the saddles and cuts $\bar F$ into a collection of polygons. 
By Lemma \ref{properties} (\ref{Disks}), each polygon is homeomorphic to a disk which lies in $B_+$ or $B_-$ away from neighborhoods of its vertices that lie in the bubbles.
Four polygons meet at each of the saddle disks of $\bar F$. 

 The Euler characteristic of $\bar F$ can be computed by summing the contribution of each disk region.
 
 \begin{lemma}\label{Polygons} 
The Euler characteristic of $\bar F$ can be computed by adding the contribution of each disk region,
with a disk region $E$ contributing $1- {s_0}/{4}$ to $\chi(\bar F)$, where 
  $s_0$ is the number of $S$'s in the word associated to the boundary of $E$.
\end{lemma}
 
 \begin{proof}
 Enumerate all curves $C_i, i=1, ..., r$, in $F \cap S_+$ and $F \cap S_-$. Suppose $C_i$ is the boundary of a polygon $E_i$ of $\bar F$ with interior in 
 $B_+$ or  $B_-$. 
 The Euler characteristic of $\bar F$ can be recovered by summing the contributions of each of these polygons $E_i, i=1, ..., r$. 
 The Euler characteristic  $\chi(\bar F) = v-e+f$ can be  distributed so that $+1/4$ is  allocated to each 
 vertex of a polygon (shared by four distinct polygons), -1/2 to each edge (shared by two distinct polygons), and +1 to each face.
 Thus the contribution of a polygon with $s_0$ vertices and the same number of edges is
 $$
 s_0/4 - s_0/2 + 1 = 1 - s_0/4.
 $$
  \end{proof}

\begin{lemma}\label{WordsLength}
Suppose $F_0$ is a closed genus $g$ surface. Any curve in $F \cap S_\pm$ has an associated word of length at most $20g-16$.
\end{lemma}
\begin{proof}
Suppose $w$ is the word associated to a curve $C \subset \bar F \cap S_+$ and $C$ is the boundary of a disk $E$ in
$\bar F \cap B_+$. Let $s_0$ be the number of saddles in $w$. By Lemma \ref{properties} (\ref{NoSSSS}), no word has a single saddle, so
we have the following cases:

1) The word $w$ consists solely of $P's$, \textit{i.e.} $s_0=0$. 
Then the curve $C$ also bounds a disk in $B_-$ and $F \cap S_+$ is a single curve, $\bar F$ is a sphere, and
the length of  $w$ is at most $4g-4$ by Lemma~\ref{numcompressions}.

2) $s_0 \ge 2$, and $E$ contributes  $1- s_0/4$ to $\chi(\bar F)$ .

In Case (2)  each curve contributes at most $+1/2$ to $\chi(\bar F)$.  Since there are at most $2g-2$ curves in $F\cap S_+$,  
any subset of the words in $F\cap S_+$ bounds polygons  that together contribute at most $g-1$ to  $\chi(\bar F)$.
 The same applies to words in $F\cap S_-$. 
Thus the  positive contribution to $\chi(\bar F)$ from any subset of  complementary polygonal disks is bounded above by $2g-2$.
Now  $ \bar F $ has at most $g-1$ components, so  $\chi(\bar F) \le 2g-2$.

Now $2-2g \leq   \chi(F) + p  =  \chi(\bar F) \le  2g-2$.  No combination of complementary polygonal disks
contributes more than  $2g-2$ to $\chi(\bar F)$, and all complementary polygonal disks together sum to $ \chi(\bar F)$.
Thus no single polygonal disk can contribute less than $ \chi(\bar F) -(2g-2)  \ge 4-4g$. A word of length
greater than $20g-16$ has at least $16g-12$ saddles, and therefore contributes less than $4-4g$ to $\chi(\bar F)$, by Lemma \ref{Polygons}.
We conclude that the length of any word is at most  $20g-16$.
\end{proof}

 \section{Bounding the number of surfaces}\label{Bound}
 
Fix, $g\geq 2$ and let $N_g(L)$ 
denote the number of  closed incompressible genus $g$ surfaces in $S^3-L$,
up to isotopy. The following theorem bounds $N_g(L)$ by a polynomial
function of the  crossing number of $L$. The argument gives explicit values for the 
constants in the bound.
 
\begin{theorem} \label{surfacebound}
Suppose a non-split reduced prime alternating link diagram $L$ has  $n$ crossings. Then 
\begin{center}
$N_g(L)\leq   C_g n^{40g^2}$, 
\end{center}
where $C_g $ is a constant that depends only on $g$.
\end{theorem}

\begin{proof}
We choose one surface $F_0$ in every isotopy class, and perform a maximal set of meridianal surgeries
on $F_0$ to obtain a surface $F$ that is in $|F|$ minimizing standard position. 
There are at most $2g-2$ words in $F\cap S_+$ by Lemma \ref{WordsNumber}, and 
each word has length at most $20g-16$ by Lemma \ref{WordsLength}. 
We give a bound on the number of non-isotopic meridianally incompressible surfaces 
by counting  all possible ways that these words can occur  in a fixed link diagram.

There are $n$ crossings, and every crossing gives $2$ choices for the location of an $S$   
adjacent to that crossing. There are $2n$ between-crossing edges in the link diagram,
giving  $2n$ choices for where to locate a puncture. 
So for each curve of length at most $20g-16$, there are at most $(4n)^{20g - 16} $ ways to choose successive saddles and punctures. 
For $2g-2$ curves in $F\cap S_+$, the total number of choices is bounded by $(4n)^{(20g - 16)(2g-2)} \le (4n)^{40g^2 } $.
 
 To bound the number of closed surfaces, we consider all ways to tube together up to $4g-4$ meridian boundary curves
 without creating self-intersections.
This number of possible tubings is computed to be ${4g-4 \choose 2g-2}$ by Mossessian in \cite{Mossessian}. 
Thus the number of closed  incompressible genus $g$ surfaces is at most
 ${4g-4 \choose 2g-2} (4n)^{40g^2 } =    {4g-4 \choose 2g-2} 4^{40g^2 } n^ {40g^2} $.  
 We let $C_g =  {4g-4 \choose 2g-2} 4^{40g^2 } $ and the Theorem follows.
\end{proof}

We can obtain the same bounds in most closed manifolds obtained by surgery on $L$
\begin{corollary}
Let $L$ be a reduced alternating hyperbolic link with $n$ crossing. The number of genus $g$ surfaces in the closed manifold obtained by $(p,q)$-surgery on $M_L$ is at most $ C_g n^{40g^2}$, with finitely many exceptional surgeries.  The number of exceptional surgeries depends only on $g$.
\end{corollary}

\begin{proof}
 The only new closed incompressible genus $g$ surfaces after  $(p,q)$-surgery are ones resulting from genus $g$ incompressible surfaces with slope $(p,q)$ in $M_L$. Theorem 4.1 in \cite{Hass:99} provides an upper bound for the number of distinct slopes $(p,q)$ that bound incompressible surfaces, as a function of the genus $g$ . The bound is given by a function $N(g)$ given in  \cite{Hass:99}. For manifolds obtained by surgeries other than this finite collection, the incompressible surfaces are incompressible in the complement of $L$.
 \end{proof}

\begin{remark}
The argument in Theorem~\ref{surfacebound}  also gives an upper bound for the number of non-isotopic embedded essential meridianal surfaces in $S^3-L$.
For such surfaces we don't need to consider all ways to tube meridians, so the bound is smaller for each genus. 
\end{remark}

\begin{remark}
If we fix a link $L$, and let the genus vary, Theorem~\ref{surfacebound} together with the explicit expression for $C_g$ in the proof provide an exponential upper bound for the number of embedded surfaces of arbitrary genus in $S^3-L$. In particular, the number grows at most as $C_L^{g^2+g}$ for a constant $C_L$ that can be written explicitly in terms of the number of crossings of $L$. While Masters and Kahn and Markovic results provide a better growth for immersed surfaces in a closed hyperbolic 3-manifold, this is the first estimate of a similar nature for surfaces in any cusped manifold (in our case, for embedded surfaces in a prime alternating link complement).
\end{remark}

%
%
%
%
%

\end{document}